\documentclass[a4paper, 11pt]{amsart}

\usepackage[dvipdfmx]{graphicx}
\usepackage{amsmath,amssymb,amsthm,mathrsfs,bm,fancybox,fullpage,dashbox}
\usepackage[all]{xy}
\usepackage{cite}
\usepackage{verbatim}
\usepackage{enumerate}
\usepackage{url}
\usepackage{color}
\usepackage{setspace}
\usepackage{multirow}
\usepackage{hyperref}

\theoremstyle{definition}
\newtheorem{theorem}{Theorem}[section]
\newtheorem{definition}[theorem]{Definition}
\newtheorem{proposition}[theorem]{Proposition}
\newtheorem{example}[theorem]{Example}

\newtheorem{lemma}[theorem]{Lemma}
\newtheorem{remark}[theorem]{Remark}
\newtheorem{question}[theorem]{Question}
\newtheorem{Main}[theorem]{Main Theorem}

\newcommand{\ra}{\rightarrow}

\newcommand{\mult}{\mathsf{mult}}

\begin{document}

\title{Uniform Cyclic Group Factorizations of Finite Groups}

\author{Kazuki Kanai}
\address{General Education Program, National Institute of Technology, Kure College, Hiroshima, 737-8506, Japan.}
\email{kanai@m.sc.niigata-u.ac.jp}
\author{Kengo Miyamoto}
\address{Department of Computer and Information Science, Ibaraki University, Ibaraki, 316-8511, Japan.}
\email{kengo.miyamoto.uz63@vc.ibaraki.ac.jp}
\author{Koji Nuida}
\address{Institute of Mathematics for Industry (IMI), Kyushu University, Fukuoka, 819-0395, Japan; National Institute of Advanced Industrial Science and Technology (AIST), Tokyo, 135-0064, Japan.}
\email{nuida@imi.kyushu-u.ac.jp}
\author{Kazumasa Shinagawa}
\address{Department of Computer and Information Science, Ibaraki University, Ibaraki, 316-8511, Japan; National Institute of Advanced Industrial Science and Technology (AIST), Tokyo, 135-0064, Japan.}
\email{kazumasa.shinagawa.np92@vc.ibaraki.ac.jp}

\thanks{K. Miyamoto is supported by Japan Society for the Promotion of Science KAKENHI 20K14302 and 23H00479.}
\thanks{K. Shinagawa, who is the corresponding author of this paper, is supported by Japan Society for the Promotion of Science KAKENHI 21K17702 and 23H00479, and JST CREST Grant Number MJCR22M1.}

\maketitle

\begin{abstract}
In this paper, we introduce a kind of decomposition of a finite group called a uniform group factorization, as a generalization of exact factorizations of a finite group.
A group $G$ is said to admit a uniform group factorization if there exist subgroups $H_1, H_2, \ldots, H_k$ such that $G = H_1 H_2 \cdots H_k$ and the number of ways to represent any element $g \in G$ as $g = h_1 h_2 \cdots h_k$ ($h_i \in H_i$) does not depend on the choice of $g$. 
Moreover, a uniform group factorization consisting of cyclic subgroups is called a uniform cyclic group factorization.
First,  we show that any finite solvable group admits a uniform cyclic group factorization.
Second, we show that whether all finite groups admit uniform cyclic group factorizations or not is equivalent to whether all finite simple groups admit uniform group factorizations or not.
Lastly, we give some concrete examples of such factorizations.
\end{abstract}

\maketitle

\section{Introduction}\label{s:introduction}

Throughout this paper, all groups are assumed to be finite and non-trivial, and the basic notations follow \cite{DM96} and Atlas \cite{ATLAS}.
For a group $G$, we denote by $P_p$ a $p$-Sylow subgroup of $G$.

A group $G$ is said to admit a factorization (resp. a group factorization) of length $k$ by an ordered tuple of subsets (resp. subgroups) $\mathcal{H}=(H_1,H_2, \ldots, H_k)$ if 
\[ G=H_1H_2\cdots H_k=\{ h_1h_2\cdots h_k\mid h_1\in H_1, h_2 \in H_2, \ldots, h_k \in H_k \}. \]
Moreover, $\mathcal{H}$ is called exact if $|G|=|H_1||H_2| \cdots |H_k|$, and is called maximal if $H_1,H_2,\ldots, H_k$ are maximal subgroups. 
Such factorizations of a group were investigated in many previous works in the literature.
A pioneering work by Miller \cite{Mil13} reveals that, 
in contrast to the fact that any group factorization $G = H_1 H_2$ of length two implies that $G = H_2 H_1$ as well (hence the ordering of the two factors does not matter), the ordering of factors is in fact essential for factorizations of length three or longer.
In particular, he showed that $A_5=P_2P_3P_5$ but $A_5\neq P_2P_5P_3$ for some Sylow subgroups $P_2,P_3$, and $P_5$.

In group theory, one of the main directions is to study simple groups.
Following the vigorous works on group factorizations of length two by Zappa \cite{Zap42} and Sz\'{e}p \cite{Sz50, Sz51}, It\^{o} \cite{Ito53} studied group factorizations of the projective special linear group $\mathrm{PSL}(2,q)$ for each prime power $q$.
In particular, he showed that there is an exact maximal group factorization by two non-normal 
subgroups of $\mathrm{PSL}(2,q)$ if $q\equiv 3 \pmod 4$ and $q>7$.
This fact together with the main theorem by \cite{Sz51} gives a short proof of the simplicity of $\mathrm{PSL}(2,q)$.
In this context, some authors worked on exact group factorizations of simple groups; for example, \cite{Gen86a}, \cite{Gen86b}, \cite{HLS87}, and so on.
Subsequently, Liebeck, Praeger, and Saxl \cite{LPS90} determined whether each sporadic simple group $G$ and its automorphism group $\mathsf{Aut}(G)$ have maximal group factorizations of length two.
After that, Giudici \cite{Giu06} determined all group factorizations of length two for sporadic simple groups.
In recent years, there have been many studies
(e.g., \cite{LX19}, \cite{LX20}, \cite{BL21}) on exact group factorizations of length two of almost simple groups (i.e., groups $G$ such that $S\leq G \leq \mathsf{Aut}(S)$ for some simple group $S$).
In 2021, Rahimipour \cite{Rah21} constructed exact group factorizations of length three or four for some sporadic simple groups.

Independent of the works described above, Magliveras \cite{Magliveras86} studied exact factorizations, where he called them logarithmic signatures, to construct a symmetric-key encryption scheme known as PGM cryptosystem. (See also \cite{MagliverasST02}, \cite{BSGM02}, \cite{Lempken09}.)
For a logarithmic signature $\mathcal{H}$, the size of $\mathcal{H}$ is defined by the sum of the cardinality of each component of $\mathcal{H}$.
Since the size of a logarithmic signature corresponds to the size of the key
in the cryptosystem, it is desirable to find a logarithmic signature of as small size as possible. 
In 2002, Gonz{\'a}lez Vasco and Steinwandt \cite{GonzalezVasco02} gave a lower bound on the size of logarithmic signatures of a finite group.
A logarithmic signature matching the lower bound is called a minimal logarithmic signature.
It has been conjectured that any finite group has a minimal logarithmic signature. 
In this context, there are various previous works on constructing minimal logarithmic signatures for finite groups; for example, \cite{GonzalezVasco03},\cite{Lempken05},\cite{Holmes04},\cite{Rahimipour18}.

The present paper introduces the notion of uniform factorizations of a finite group as an analogy of logarithmic signatures (or exact factorizations). 
Let $G$ be a finite group and $\mathcal{H} = (H_1, H_2, \ldots, H_k)$ a tuple of subsets of $G$.
If, for any $g\in G$, the number of tuples $(h_1, h_2, \ldots, h_k) \in H_1 \times H_2 \times \cdots \times H_k$ with $g=h_1h_2\cdots h_k$ does not depend on the choice of $g$, then $\mathcal{H}$ is called a uniform factorization of $G$
and we write $G \equiv H_1H_2\cdots H_k$.
Obviously, the tuple $(G)$ of length one is a uniform factorization.
If $H_1, H_2,\ldots, H_k$ are proper subsets, the factorization is said to be proper. 
If $H_1,H_2,\ldots, H_k$ are subgroups of $G$, then $\mathcal{H}$ is called a uniform group factorization of $G$. 
Moreover, if they are cyclic subgroups of $G$, then $\mathcal{H}$ is called a uniform cyclic group factorization of $G$.

Uniform cyclic group factorizations for various finite groups are expected to have applications in computer algebra. 
For example, they can be applied to efficient generation of uniformly random elements of a finite group $G$. 
A straightforward method involves assigning an integer from $1$ to $|G|$ to each element and selecting the $x$-th element for a uniformly random number $x \in \{1, 2, \ldots, |G|\}$. 
Although this method is feasible when the elements of $G$ are efficiently enumerable, in general, this method requires storing all elements in a table, which requires a huge storage space when $|G|$ is large.
In contrast, if a uniform cyclic group factorization $G \equiv H_1H_2\cdots H_k$ exists, such a random element $g \in G$ can be generated by $g= h_1^{x_1}h_2^{x_2} \cdots h_k^{x_k}$ where $h_i$ is a generator of $H_i$ and $x_i \in \{0, 1, \ldots, |H_i|-1\}$ is chosen uniformly at random.
This method only requires storing $k$ elements $h_1,h_2,\dots,h_k$, significantly reducing the storage space. 
Here we emphasize that, in order to make the element $g$ uniformly random,  it is \emph{not} necessary for a decomposition of $g$ into elements of $H_i$'s being \emph{unique}. The requirement is that there are a constant number of such decompositions independent of $g$, justifying our relaxed condition for uniform factorizations compared to logarithmic signatures.
We note that there is a line of studies (e.g., \cite{D08}) on the problem of generating group elements with a distribution close to uniform, while the above method based on uniform cyclic group factorizations generates perfectly uniform distribution.

Now it is natural to ask the following question:
\begin{question}\label{question}
Does any finite group have a uniform cyclic group factorization? 
\end{question}

One of the two main results of this paper asserts that Question \ref{question} can be reduced to the case of non-solvable groups.  Precisely, we have the following theorem:

\begin{Main}[{Theorem \ref{theorem:solvable}}]
Any finite solvable group admits a uniform cyclic group factorization.
In particular, any finite abelian group admits a uniform cyclic group factorization.
\end{Main}

The other main result of this paper asserts that Question 1.1 can be further reduced to the case of uniform (not necessarily cyclic) group factorizations for simple groups. Precisely, we have the following theorem:

\begin{Main}[Theorem \ref{theorem:simple}]
Let $n$ be a positive integer. 
The following are equivalent.
\begin{enumerate}
    \item[(a)] Any $G \in \mathscr{G}_n$ admits a uniform cyclic group factorization, where $\mathscr{G}_n$ is the set of groups of order at most $n$.
    \item[(b)] Any $G \in \mathscr{G}_n^{\dagger}$ admits a proper uniform group factorization, where $\mathscr{G}_n^{\dagger}$ consists of groups in $\mathscr{G}_n$ not cyclic of prime power.
    \item[(c)] Any $G \in \mathscr{G}_n^*$ admits a proper uniform group factorization, where $\mathscr{G}_n^*$ consists of simple groups in $\mathscr{G}_n^{\dagger}$.
\end{enumerate}
\end{Main}
This paper consists of five sections.
In Section 2, we introduce the notion of uniform cyclic group factorizations. In Section 3, we prove the main theorems described above.
In Section 4, we discuss relationships among logarithmic signatures and uniform cyclic group factorizations of finite groups.
In Section 5, we construct some concrete examples of uniform group factorizations by following methods of constructing logarithmic signatures (e.g., \cite{GonzalezVasco03}).
In particular, we give uniform cyclic group factorizations of the alternating groups.

\section{Uniform cyclic group factorizations}\label{s:notation}

Let $G$ be a finite group and $\mathcal{H} = (H_1, H_2, \ldots, H_k)$ an ordered tuple of subsets of $G$. 
Define the multiplication map $\mult_{\mathcal{H}}: H_1 \times H_2\times \cdots \times H_k \ra G$ by
\[
\mult_{\mathcal{H}}(h_1, h_2, \ldots, h_k) := h_1h_2 \cdots h_k.
\]
Then, $\mathcal{H}$ is called a \textit{factorization} of $G$ 
if $\mult_{\mathcal{H}}$ is surjective.
The integer $k$ is called the \textit{length} of $\mathcal{H}$.
We write $G = H_1H_2\cdots H_k$ when $\mathcal{H}$ is a factorization of $G$.
If all $H_1,H_2,\ldots, H_k$ are proper subsets of $G$, then $\mathcal{H}$ is called a \textit{proper factorization} of $G$. 

\begin{definition}
Let $G$ be a finite group and $\mathcal{H}=(H_1,H_2,\ldots, H_{k})$ a factorization of $G$.
\begin{enumerate}
\item[(1)] The factorization $\mathcal{H}$ is a \textit{uniform factorization} of $G$ if $|\mult_{\mathcal{H}}^{-1}(g)|$ does not depend on $g \in G$. 
The integer $t := |\mult_{\mathcal{H}}^{-1}(g)|$ is called the \textit{multiplicity} of $\mathcal{H}$. 
In this case, we write $G \overset{t}{\equiv} H_1H_2\cdots H_k$, or simply $G \equiv H_1H_2\cdots H_k$. 
\item[(2)] The factorization $\mathcal{H}$ is a \textit{uniform group factorization} of $G$ if $\mathcal{H}$ is a uniform factorization of $G$ and all $H_1,H_2,\ldots, H_k$ are subgroups of $G$.
\item[(3)] The factorization $\mathcal{H}$ is a \textit{uniform cyclic group factorization} of $G$ if $\mathcal{H}$ is a uniform group factorization of $G$ and all $H_1,H_2,\ldots, H_k$ are cyclic subgroups of $G$.
\end{enumerate}
\end{definition}

\begin{remark}\label{rem:direct_product}
If a finite group $G$ admits a direct product decomposition $G = H_1 \times H_2 \times\cdots \times H_k$ into (normal) subgroups $H_1,H_2,\dots,H_k$, then it follows immediately that $(H_1,H_2,\dots,H_k)$ is a uniform group factorization of $G$ with multiplicity one.
\end{remark}

The following lemma enables recursive construction of uniform group factorizations.

\begin{lemma}\label{lem:arrange}
Let $G$ be a group and $\mathcal{H} = (H_1,H_2, \ldots, H_k)$ a uniform group factorization of $G$ with multiplicity $t$. 
If each $H_i$ ($1 \leq i \leq k$) has a uniform cyclic group factorization with multiplicity $t_i$, then $G$ also has a uniform cyclic group factorization with multiplicity $t \cdot \prod_{i=1}^kt_i$.
\end{lemma}

\begin{proof} 
Let $\mathcal{H}_i = (H_{i,1}, \ldots, H_{i,\ell_i})$ be a uniform cyclic group factorization of $H_i$ for $1 \leq i \leq k$. 
Let $\mathcal{H}'$ be the ordered tuple of cyclic groups in the following:
\[
\mathcal{H}' := (H_{1,1}, \ldots, H_{1,\ell_1}, \ldots, H_{k,1}, \ldots, H_{k,\ell_k}).
\]
Since $\mult_{\mathcal{H'}}^{-1}(g)$ for any $g \in G$ is expressed as
\[
 \bigsqcup_{(h_1,\ldots,h_k) \in \mult_{\mathcal{H}}^{-1}(g)} \left\{(h_{1,1}, \ldots, h_{1,\ell_1}, \ldots, h_{k,1}, \ldots, h_{k,\ell_k}) \mid (h_{i,1}, \ldots, h_{i,\ell_i}) \in \mult_{\mathcal{H}_i}^{-1}(h_i)\right\},
\]
we have $|\mult_{\mathcal{H'}}^{-1}(g)| = t\cdot \prod_{i=1}^kt_i$, which proves the statement. 
\end{proof}

\begin{example}
For a finite group $G$, uniform cyclic group factorizations are not necessarily unique.
In fact, the symmetric group $S_5$ has the following three uniform cyclic group factorizations: 
\begin{align*}
&\mathcal{H}_1 = \left(\langle (1,2) \rangle, \langle (1,2,3) \rangle, \langle (1,2,3,4) \rangle, \langle (1,2,3,4,5) \rangle \right),\\
&\mathcal{H}_2= \left(\langle (1,2,3,4,5)\rangle, \langle(1,2,4,3)\rangle,\langle(1,2,3)(4,5)\rangle \right),\\
&\mathcal{H}_3= \left(\langle (1,2,3,4,5)\rangle, \langle(1,2,3,4)\rangle,\langle(1,3,2,4)\rangle, \langle (1,2,3) \rangle \right).
\end{align*}
The multiplicities of $\mathcal{H}_1$ and $\mathcal{H}_2$ are $1$ since the product of the cardinality of each subgroup equals $|S_5|$. The multiplicity of $\mathcal{H}_3$ is $2$ since the product of the cardinality of each subgroup equals $2 \cdot |S_5|$.
\end{example}

The following lemma is fundamental.

\begin{lemma}\label{lemma:hom}
Let $G$ and $G'$ be two finite groups, and $f: G\to G'$ a group homomorphism.
Let $H_1, H_2, \ldots, H_k$ be subgroups of $G$. Define $t_i := |\mathsf{ker}(f) \cap H_i|$. If $\mathcal{H}' := (f(H_1), f(H_2), \ldots, f(H_k))$ is a uniform group factorization of $G'$ with multiplicity $t'$, then
\[
\mathcal{H} = (H_1, H_2, \ldots, H_k, \mathsf{ker}(f))
\] 
is a uniform group factorization of $G$ with multiplicity $t'\cdot \prod_{i=1}^kt_i$.
\end{lemma}

\begin{proof}
For any $g \in G$, we define a map 
$\phi:\mult_{\mathcal{H}}^{-1}(g) \to \mult^{-1}_{\mathcal{H}'}(f(g))$ 
by sending $(h_1,h_2,\ldots, h_k, x)\in\mult_{\mathcal{H}}^{-1}(g)$ 
to 
$(f(h_1),f(h_2),\ldots, f(h_k)) \in \mult^{-1}_{\mathcal{H'}}(f(g))$.
The map $\phi$ is surjective because for any $(h'_1, h'_2, \ldots, h'_k) \in \mult^{-1}_{\mathcal{H'}}(f(g))$, by taking $h_i \in H_i$ such that $f(h_i) = h'_i$, we have $(h_1 h_2 \cdots h_k)^{-1}g \in \mathsf{ker}(f)$ and $\phi(h_1, h_2, \ldots, h_k, (h_1 h_2 \cdots h_k)^{-1}g) = (h'_1, h'_2, \ldots, h'_k)$.
Then we have
\[
\mult_{\mathcal{H}}^{-1}(g) = \bigsqcup_{(h_1',h_2', \ldots, h_k') \in \mult_{\mathcal{H}'}^{-1}(f(g))} \phi^{-1}(h_1',h_2', \ldots, h_k').
\]
Let $h' = (h_1',h_2', \ldots, h_k') \in \mult_{\mathcal{H}'}^{-1}(f(g))$.
Put 
\[
A_{h'} := (f^{-1}(h_1') \cap H_1) \times (f^{-1}(h_2') \cap H_2) \times \cdots \times (f^{-1}(h_k') \cap H_k) \subseteq H_1\times H_2\times \cdots \times H_k.
\]
Since the map $\psi_{h'}: A_{h'}\to \phi^{-1}(h')$ defined by 
\[ \psi_{h'}((h_1, \ldots, h_k)) = (h_1, \ldots, h_k, h_k^{-1} \cdots h_1^{-1}g) \]
is bijective and $|f^{-1}(h_i')\cap H_i| = |\mathsf{ker}(f) \cap H_i|=t_i$, it implies that $|\phi^{-1}(h')| = \prod_{i=1}^kt_i$. 
Thus we have
\[
|\mult_{\mathcal{H}}^{-1}(g)| = \sum_{h' \in \mult_{\mathcal{H}'}^{-1}(f(g))} |\phi^{-1}(h')| = |\mult_{\mathcal{H}'}^{-1}(f(g))|\cdot \prod_{i=1}^kt_i = t'\cdot \prod_{i=1}^kt_i.
\]
Therefore, $\mathcal{H}$ is a uniform group factorization of $G$ with multiplicity $t' \cdot \prod_{i=1}^kt_i$.
\end{proof}

\begin{proposition}\label{prop:CyclicOfPrimeOrder}
Let $G$ be a cyclic group.
The following are equivalent: 
\begin{enumerate}
\item[(i)] $G$ admits a proper uniform group factorization.
\item[(ii)] $|G|$ is not a prime power.
\end{enumerate}
Moreover, if these conditions hold, then the factorization in (i) can be made cyclic and with multiplicity one.
\end{proposition}

\begin{proof}
Let $\gamma$ be a fixed generator of $G$.
Set $n := |G|$. 

To show that negation of (ii) implies negation of (i), assume that $n = p^r$ for some prime number $p$. 
Then any proper (and nontrivial) subgroup of $G$ is a cyclic group of the form $\langle \gamma^{p^{r-r'}}\rangle$, where $r > r' \geq 1$.
Thus, any tuple of proper subgroups $(H_1, H_2, \ldots, H_k)$ can not be a factorization since $\gamma \not\in H_1H_2 \cdots H_k$.
Therefore, $G$ has no proper uniform group factorization.

To show that (ii) implies (i), assume that $n$ is not a prime power.
In this case, $n$ can be written as $n = n_1 n_2$ with $n_1, n_2 > 1$ being coprime. 
Now we have $G \simeq \mathbb{Z}/n\mathbb{Z} \simeq \mathbb{Z}/n_1\mathbb{Z} \times\mathbb{Z}/n_2\mathbb{Z}$ by Chinese Remainder Theorem, therefore $G$ admits a proper uniform cyclic group factorization with multiplicity one by Remark \ref{rem:direct_product}.
\end{proof}

\section{Main results}\label{s:proofs}

The aim of this section is to prove the main results of this paper.

First, we mention that the next assertion follows from Lemma \ref{lemma:hom} immediately.

\begin{lemma}\label{lemma:normal}
Let $G$ be a finite group with a proper normal subgroup $N$, and $\pi: G \to G/N$ the canonical surjection. 
Let $H_1, H_2, \ldots, H_k$ be subgroups of $G$. If $\mathcal{H}' := (\pi(H_1), \pi(H_2), \ldots, \pi(H_k))$ is a uniform group factorization of $G/N$, then
\[ \mathcal{H} := (H_1,H_2, \ldots, H_k, N) \]
is a uniform group factorization of $G$.
Moreover, if $\mathcal{H}'$ is a proper uniform group factorization, then so is $\mathcal{H}$.
\end{lemma}

The next lemma is easy but is useful in proving our first main theorem.

\begin{lemma}\label{lemma:Abel}
Any finite abelian group admits a uniform cyclic group factorization.
\end{lemma}

\begin{proof}
The structure theorem for finite abelian groups implies that the group in the statement is a direct product of cyclic subgroups.
Then the claim follows from Remark \ref{rem:direct_product}.
\end{proof}

Now we give the first main theorem of this paper.

\begin{theorem}\label{theorem:solvable}
Any finite solvable group admits a uniform cyclic group factorization.
\end{theorem}
\begin{proof}
Let
\[
\{1\}=G_{\ell} \leq G_{\ell-1}\leq \cdots  \leq G_1\leq G_0=G
\]
be a subnormal series of finite length with strict inclusions such that $G_{i}/G_{i+1}$ is an abelian group for $0\leq i\leq \ell-1$.
From Lemma \ref{lemma:Abel}, $G_{i}/G_{i+1}$ admits a uniform cyclic group factorization, say $\mathcal{H}'_i = (H'_{i,1}, \ldots, H'_{i,k_i})$ for $0\leq i\leq \ell-1$. 
Now by taking a preimage of a generator of $H'_{i,j}$ ($j = 1,\dots,k_i$), we can construct a cyclic subgroup $H_{i,j}$ of $G_i$ with $\pi(H_{i,j}) = H'_{i,j}$.
Then by Lemma \ref{lemma:normal}, $\mathcal{H}_i = (H_{i,1}, \ldots, H_{i,k_i}, G_{i+1})$ is a uniform group factorization of $G_i$ for $0 \leq i \leq \ell-1$.
By applying Lemma \ref{lem:arrange} recursively for $i = \ell-1,\ell-2,\dots,0$, it follows that
\[
\mathcal{H}_i := (H_{i,1}, \ldots, H_{i,k_i}, H_{i+1,1}, \ldots, H_{i+1,k_{i+1}},  \ldots, H_{\ell-1,1}, \ldots, H_{\ell-1,k_{\ell-1}})
\]
is a uniform cyclic group factorization of $G_i$.
Now $\mathcal{H}_0$ is the factorization of $G = G_0$ as in the statement.
\end{proof}

For a positive integer $n$, we define three sets $\mathscr{G}_n$, $\mathscr{G}_n^{\dagger}$, and $\mathscr{G}_n^*$ as follows. 
\begin{itemize}
\item $\mathscr{G}_n$ is the set consisting of isomorphism classes of finite groups of order at most $n$. 
\item $\mathscr{G}_n^{\dagger}$ is the subset of $\mathscr{G}_n$ obtained by removing cyclic groups of prime power.
\item $\mathscr{G}_n^*$ is the subset of $\mathscr{G}_n^{\dagger}$ obtained by removing non-simple groups.
\end{itemize} 
By definition, we have the following inclusions:
\[
\mathscr{G}_n^* \subseteq \mathscr{G}_n^{\dagger} \subseteq \mathscr{G}_n.
\]

Now we give the second main theorem of this paper.

\begin{theorem}\label{theorem:simple}
Let $n\geq 1$ be an integer. 
The following are equivalent. 
\begin{enumerate}
    \item[(a)] Any $G \in \mathscr{G}_n$ admits a uniform cyclic group factorization.
    \item[(b)] Any $G \in \mathscr{G}_n^{\dagger}$ admits a proper uniform group factorization.
    \item[(c)] Any $G \in \mathscr{G}_n^*$ admits a proper uniform group factorization. 
\end{enumerate}
\end{theorem}

\begin{proof}
(a) $\Longrightarrow$ (b):
Let $G \in \mathscr{G}_n^{\dagger}$.
If $G$ is not cyclic, it has a uniform cyclic group factorization by (a). Since $G$ is not cyclic, the factorization is proper. If $G$ is a cyclic group whose order is not a prime power, it admits a proper uniform group factorization by Proposition \ref{prop:CyclicOfPrimeOrder}.
Therefore, (a) implies (b).

(b) $\Longrightarrow$ (c): This implication is trivial. 

(c) $\Longrightarrow$ (a):
We show the assertion by induction on $n$.
For a positive integer $n$, we suppose that (c) holds.
Then the condition (c) for $\mathscr{G}_{n-1}^*$ also holds, therefore the induction hypothesis implies that the condition (a) for $\mathscr{G}_{n-1}$ holds as well.
Let $G \in \mathscr{G}_n$. If $G$ is a cyclic group (including the base case $n = 1$), then $G$ itself can be regarded as a uniform cyclic group factorization of $G$.
Thus, we may assume that $G$ is not cyclic.
If $G$ is a simple group, then $G$ admits a proper uniform group factorization, say $\mathcal{H}$, by (c).
Since (a) holds for $\mathscr{G}_{n-1}$ as mentioned above, each component of $\mathcal{H}$ admits a uniform cyclic group factorization, therefore $G$ also admits a uniform cyclic group factorization by Lemma \ref{lem:arrange}.
If $G$ is not a simple group, then a maximal normal subgroup of $G$ exists, say $N$.
Since $|G/N| < |G| \leq n$ and (a) holds for $\mathscr{G}_{n-1}$ as mentioned above, $G/N$ admits a uniform cyclic group factorization, say $\mathcal{H}'$.
From this, we can obtain a uniform group factorization $\mathcal{H}$ of $G$ as in Lemma \ref{lemma:normal}, where any component other than $N$ can be chosen as being cyclic (as well as those in $\mathcal{H}'$).
Now since $|N| < |G|$ and (a) holds for $\mathscr{G}_{n-1}$ as mentioned above, $N$ admits a uniform cyclic group factorization.
Hence by Lemma \ref{lem:arrange}, $G$ also admits a uniform cyclic group factorization.

Therefore, (c) implies (a).
\end{proof}

\section{Logarithmic signatures and uniform cyclic group factorizations}\label{s:LS}

Let $G$ be a finite group, and $\mathcal{H} = (H_1, H_2, \ldots, H_k)$ a tuple of subsets of $G$. 
The tuple $\mathcal{H}$ is called a \textit{logarithmic signature} (or an \textit{exact factorization}) of $G$, which is named by \cite{Magliveras86}, if $\mathcal{H}$ is a uniform factorization of $G$ with multiplicity one. 
If $\mathcal{H}$ is a logarithmic signature, the \textit{size} of $\mathcal{H}$ is defined by $\ell(\mathcal{H}) := |H_1| + |H_2| + \cdots + |H_k|$. 

Gonz{\'a}lez Vasco and Steinwandt~\cite{GonzalezVasco02} gave a lower bound on the size of logarithmic signatures.
The lower bound is given as follows.
Suppose that $|G| = \prod_{i=1}^m p_i^{a_i}$, where the $p_i$'s are distinct prime numbers and $a_i$ is a positive integer. 
Then they showed that the following inequality holds for any logarithmic signature $\mathcal{H}$:
\[
\ell(\mathcal{H}) \geq \sum_{i=1}^m a_i p_i.
\]
If the equality holds, $\mathcal{H}$ is called a \textit{minimal logarithmic signature} of $G$. 

Let $G$ be a finite group. 
We define the following sets:
\begin{align*}
&\mathsf{UF}(G) := \text{the set of uniform factorizations of $G$.}\\
&\mathsf{UGF}(G) := \text{the set of uniform group factorizations of $G$.}\\
&\mathsf{UCF}(G) := \text{the set of uniform cyclic group factorizations of $G$.}\\
&\mathsf{LS}(G) := \text{the set of logarithmic signatures of $G$.}\\
&\mathsf{LGS}(G) := \mathsf{LS}(G) \cap \mathsf{UGF}(G)\\
&\mathsf{LCS}(G) := \mathsf{LS}(G) \cap \mathsf{UCF}(G)\\
&\mathsf{MLS}(G) := \text{the set of minimal logarithmic signatures of $G$.}\\
&\mathsf{MLGS}(G) := \mathsf{MLS}(G) \cap \mathsf{UGF}(G)\\
&\mathsf{MLCS}(G) := \mathsf{MLS}(G) \cap \mathsf{UCF}(G)
\end{align*}

By definition, we have relations of inclusion among these sets. 
\[
\begin{tabular}{ccccc}
$\mathsf{UF}(G)$ & $\supseteq$ & $\mathsf{UGF}(G)$ & $\supseteq$ & $\mathsf{UCF}(G)$ \\
\rotatebox{90}{$\subseteq$} & & \rotatebox{90}{$\subseteq$} & & \rotatebox{90}{$\subseteq$}\\
$\mathsf{LS}(G)$ & $\supseteq$ & $\mathsf{LGS}(G)$ & $\supseteq$ & $\mathsf{LCS}(G)$\\
\rotatebox{90}{$\subseteq$} & & \rotatebox{90}{$\subseteq$} & & \rotatebox{90}{$\subseteq$}\\
$\mathsf{MLS}(G)$ & $\supseteq$ & $\mathsf{MLGS}(G)$ & $\supseteq$ & $\mathsf{MLCS}(G)$\\
\end{tabular}
\]
The next proposition shows that these notions are in fact distinct.

\begin{proposition}\label{lemma:separation}
The following statements hold:
\begin{enumerate}
\item[(1)] There exists $G$ such that $\mathsf{MLS}(G) \supsetneq \mathsf{MLGS}(G)$, $\mathsf{LS}(G) \supsetneq \mathsf{LGS}(G)$, and $\mathsf{UF}(G) \supsetneq \mathsf{UGF}(G)$. 
\item[(2)] There exists $G$ such that $\mathsf{MLGS}(G) \supsetneq \mathsf{MLCS}(G)$, $\mathsf{LGS}(G) \supsetneq \mathsf{LCS}(G)$, and $\mathsf{UGF}(G) \supsetneq \mathsf{UCF}(G)$. 
\item[(3)] There exists $G$ such that $\mathsf{UCF}(G) \supsetneq \mathsf{LCS}(G)$, $\mathsf{UGF}(G) \supsetneq \mathsf{LGS}(G)$, and $\mathsf{UF}(G) \supsetneq \mathsf{LS}(G)$. 
\item[(4)] There exists $G$ such that $\mathsf{LCS}(G) \supsetneq \mathsf{MLCS}(G)$, $\mathsf{LGS}(G) \supsetneq \mathsf{MLGS}(G)$, and $\mathsf{LS}(G) \supsetneq \mathsf{MLS}(G)$. 
\end{enumerate}
\end{proposition}

\begin{proof}
(1) Let $G = C_4 = \langle \sigma \rangle$. 
Set $\mathcal{H} := (H_1, H_2)$ for $H_1 = \{1, \sigma\}$ and $H_2 = \{1, \sigma^2\}$. 
Since $G \overset{1}{\equiv} H_1H_2$ and $\ell(\mathcal{H})= 4 = 2 \cdot 2$, 
we have $\mathcal{H} \in \mathsf{MLS}(G)$, therefore $\mathcal{H} \in \mathsf{LS}(G)$ and $\mathcal{H} \in \mathsf{UF}(G)$.
On the other hand, since $H_1$ is not a group, we have $\mathcal{H} \not\in \mathsf{UGF}(G)$, therefore $\mathcal{H} \not\in \mathsf{LGS}(G)$ and $\mathcal{H} \not\in \mathsf{MLGS}(G)$.

(2) Let $G = C_2 \times C_2 \times C_2 = \langle \sigma_1 \rangle \times  \langle \sigma_2 \rangle \times  \langle \sigma_3 \rangle$. 
Set $\mathcal{H} := (H_1, H_2)$ for $H_1 = \langle \sigma_1, \sigma_2 \rangle$ and $H_2 = \langle \sigma_3 \rangle$. 
Since $G \overset{1}{\equiv} H_1H_2$ and $\ell(\mathcal{H})= 6 = 3 \cdot 2$, 
we have $\mathcal{H} \in \mathsf{MLGS}(G)$, therefore $\mathcal{H} \in \mathsf{LGS}(G)$ and $\mathcal{H} \in \mathsf{UGF}(G)$.
On the other hand, since $H_1$ is not cyclic, we have $\mathcal{H} \not\in \mathsf{UCF}(G)$, therefore $\mathcal{H} \not\in \mathsf{LCS}(G)$ and $\mathcal{H} \not\in \mathsf{MLCS}(G)$.

(3) Let $G$ be any cyclic group, and $\mathcal{H} := (G, G)$. 
Then we have $\mathcal{H} \in \mathsf{UCF}(G)$, therefore $\mathcal{H} \in \mathsf{UGF}(G)$ and $\mathcal{H} \in \mathsf{UF}(G)$.
On the other hand, we have $\mathcal{H} \not\in \mathsf{LS}(G)$, therefore $\mathcal{H} \not\in \mathsf{LGS}(G)$ and $\mathcal{H} \not\in \mathsf{LCS}(G)$.

(4) Let $n > 1$ be an integer which is neither $4$ nor a prime number. 
Let $G$ be a cyclic group of order $n$. 
Set $\mathcal{H} := (G)$. 
Then we have $\mathcal{H} \in \mathsf{LCS}(G)$, therefore $\mathcal{H} \in \mathsf{LGS}(G)$ and $\mathcal{H} \in \mathsf{LS}(G)$.
On the other hand, we have $\mathcal{H} \not\in \mathsf{MLS}(G)$, therefore $\mathcal{H} \not\in \mathsf{MLGS}(G)$ and $\mathcal{H} \not\in \mathsf{MLCS}(G)$.

The proof is completed. 
\end{proof}
 
There is a line of research on the following question: Does every finite group have a minimal logarithmic signature?
On the other hand, our question in this paper is: Does every finite group have a uniform cyclic group factorization?
Since there is no inclusion relation between $\mathsf{MLS}(G)$ and $\mathsf{UCF}(G)$ in general, these questions are independent. 
Compared to (minimal) logarithmic signatures, uniform cyclic group factorizations are more restrictive from the viewpoint that each $H_i$ is restricted to a cyclic group, while it is less restrictive from the viewpoint that the multiplicity may be larger than one.

The former viewpoint leads to the fact that not every existing construction of (minimal) logarithmic signatures yields uniform (cyclic) group factorizations (e.g., the construction method of double coset decomposition~\cite{Holmes04}, \cite{Lempken05}).
We note that some of them are indeed useful for constructing uniform (cyclic) group factorizations. 
For example, \cite[Theorem 3.1]{Lempken05} gives a minimal logarithmic signature of $PSL(n,q)$ ($n \geq 2,\,  \mathrm{gcd}(n,q-1) = 1$) which yields a uniform group factorization of it.

The latter viewpoint allows for other construction methods on uniform cyclic group factorizations which are not applicable to (minimal) logarithmic signatures. 
In particular, since $G \equiv H_1H_2$ and $G = H_1H_2$ are equivalent when the length of factorization is two, a uniform group factorization of $G$ is immediately obtained from $G = H_1H_2$ (see Section \ref{ss:two}). 

\section{Examples}\label{s:construction}

\subsection{Sylow systems}\label{ss:Sylow}

Let $G$ be a finite group, and $\pi(G) = \{p_1, p_2, \ldots, p_{\ell}\}$ the prime factors of $|G|$.
For any $p_i\in \pi(G)$, we take a Sylow subgroup $P_{p_i}$ of $G$.
Then the ordered tuple $(P_{p_1}, P_{p_2}, \ldots, P_{p_{\ell}})$ is called a \textit{Sylow system} of $G$ if $(P_{p_{\sigma(1)}}, P_{p_{\sigma(2)}}, \ldots, P_{p_{\sigma(\ell)}})$ is a uniform group factorization with multiplicity $1$ for any permutation $\sigma$.
It is well-known that any finite solvable group has a Sylow system  \cite[Subsection 6.4, Theorem 4.3]{Gorenstein07}.

There are the following three cases:
\begin{enumerate}
\renewcommand{\labelenumi}{(\Roman{enumi})}
\item $G$ has a Sylow system.
\item $G$ does not have a Sylow system, but has a uniform group factorization with multiplicity $1$ consisting of Sylow subgroups.
\item $G$ has neither a Sylow system nor a uniform group factorization with multiplicity $1$ consisting of Sylow subgroups.
\end{enumerate}

Some researchers have studied which finite groups belong to which type; for example, see Table \ref{table:sylow}.
Since groups classified as (I) and (II) have uniform group factorizations, it is an important research question to make it clear which non-solvable groups belong to Types (I) or (II).

\begin{table}[htb]
  \caption {Sylow systems and uniform group factorizations of finite groups}
  \label {table:sylow}
\centering
    \begin{tabular}{c|c}\hline
    Types & Finite groups \\ \hline \hline 
    (I) & Solvable groups \cite{Gorenstein07}, $L_3(2)$ \cite{Mil13}  \\ \hline
    \multirow{2}*{(II)} & $A_5$ \cite{Mil13}, $PGL(2,q)$, $PSL(2,q)$ ($q \not\equiv 1 \bmod 3$), $A_7, A_8$ \cite{Holt93} \\ 
         & $\mathrm{PSL}_3(4), \mathrm{PSU}_4(2), \mathrm{PSU}_3(4)$ \cite{GonzalezVasco03} \\ \hline
    (III) &  $U_3(3)$ \cite{Holt93} \\ \hline
    \end{tabular}
\end{table}

\subsection{Alternating groups}\label{ss:max}

Let $\{G_n\}_{n\in\mathbb{Z}}$ be a family of finite groups such that $G_n$ acts on a set $X_n$.
Then, $\{G_n\}_{n\in\mathbb{Z}}$ is a \textit{stabilizer chain} of $\{X_n\}_{n\in\mathbb{Z}}$ if, for any $n\in\mathbb{Z}$, there exists $x_n\in X_n$ such that the stabilizer $\mathsf{Stab}_{G_n}(x_n)$ is isomorphic to $G_{n-1}$.
Gonz{\'a}lez Vasco, R\"{o}tteler, and Steinwandt constructed a uniform group factorization of each Mathieu group with multiplicity $1$ based on stabilizer chains; for details, see \cite{GonzalezVasco03}.
Their method can be also extended to the case of some other groups.
In this subsection, as an example, we give a uniform cyclic group factorization of the alternating group $A_n$ $(n\geq 3)$.
First, we show the following useful lemma.

\begin{lemma}\label{lem:subgroups}
Let $G$ be a finite group, and $H, K_1, K_2, \ldots, K_{\ell}$ $(\ell\geq 1)$ non-trivial subgroups of $G$.
Assume that the following conditions (a) and (b) are satisfied.
\begin{enumerate}
\item[(a)] $|G| = |H|\cdot \displaystyle{\prod_{i=1}^{\ell}|K_i|}$.
\item[(b)] For any $g \in G$, there exists $(k_1,k_2,\ldots, k_\ell) \in K_1\times K_2\times\cdots\times K_{\ell}$ such that $k_{1} k_{2} \cdots k_{\ell}g \in H$.
\end{enumerate}
Then, $(K_{\ell}, K_{\ell-1}, \ldots, K_{1}, H)$ is a proper uniform group factorization of $G$ with multiplicity $1$.
\end{lemma}
\begin{proof}
Let $g\in G$. By condition (b), there exists $(k_1,k_2,\ldots, k_\ell) \in K_1\times K_2\times\cdots\times K_{\ell}$ such that $h:=k_{1} k_{2} \cdots k_{\ell}g \in H$.
Then, we have $g = k_{\ell}^{-1}k_{\ell-1}^{-1}\cdots k_1^{-1}h\in K_{\ell}K_{\ell-1}\cdots K_{1}H$.
By condition (a), such expression is unique.
Therefore,  $(K_{\ell}, K_{\ell-1}, \ldots, K_{1}, H)$ is a proper uniform group factorization of $G$ with multiplicity $1$.
\end{proof}

Now, we construct a uniform cyclic group factorization of the alternating group $A_n$.

\begin{proposition}
For any integer $n\geq 3$, the alternating group $A_n$ admits a uniform cyclic group factorization with multiplicity $1$.
\end{proposition}
\begin{proof}
We show the statement by induction on $n$.

If $n=3$, the assertion follows obviously since $A_3$ is cyclic (of order three).

Suppose that $n>3$.
We consider the natural action of $A_n$ on $\{1,2,\ldots, n\}$.
Take $H = \mathsf{Stab}_G(n)\simeq A_{n-1}$, which acts on $\{1,2,\ldots, n-1\}$ naturally.

First, we suppose that $n$ is odd.
We consider the subgroup $K$ of $A_n$ generated by $(1,2,3,\ldots,n)\in A_n$.
Then, for any $\rho\in A_n$, we observe that $(1,2,3,\ldots,n)^{n-\rho(n)}\rho$ fixes the point $n$, that is, 
\[ (1,2,3,\ldots,n)^{n-\rho(n)}\rho\in H. \]
Thus, it follows from Lemma \ref{lem:subgroups} that $A_n \overset{1}{\equiv} KH$ is a uniform group factorization.
By induction hypothesis, $H$ has a uniform cyclic group factorization, and so does $A_n$ by Lemma \ref{lem:arrange}.

Now, we suppose that $n=2m$ for some positive integer $m$.
We put 
\[\sigma_1 := (1,2,\ldots , m)(m+1,m+2,\ldots ,2m),\quad  \sigma_2 := (1,2)(m,2m). \]
Then, it is easy to check that $\sigma_1$ and $\sigma_2$ belong to $A_n$.
Let $K_1 = \langle \sigma _1 \rangle$ and $K_2 = \langle \sigma_2 \rangle$.
Let $\rho \in A_n$.
If $1 \leq \rho(n) \leq m$, then $\sigma_2\sigma_1^{m-\rho(n)}\rho$ fixes the point $n$.
Otherwise, $e\sigma_1^{2m-\rho(n)}\rho$ fixes the point $n$, where $e$ is the identity element of $K_2$.
Thus, $A_n \overset{1}{\equiv} K_1K_2H$ is a uniform group factorization of $A_n$ by Lemma \ref{lem:subgroups}.
By induction hypothesis, $H$ has a uniform cyclic group factorization, and so does $A_n$ by Lemma \ref{lem:arrange}.
\end{proof}

\begin{remark}\label{rem:Magliveras}
A construction of a uniform group factorization of $A_n$ can be found in \cite{Magliveras02}. However, the construction needs to be corrected.
Indeed, if $n=2m+1$ for some odd integer $m\geq 1$, then the group 
\[ \langle (1,2,\ldots, m)(m+1,m+2,\ldots ,2m), \ (1,m+1)(2,m+2)\cdots(m,2m) \rangle  \]
appeared in the construction by \cite{Magliveras02} is not a subgroup of $A_n$.
\end{remark}

\subsection{Group factorizations of length two}\label{ss:two}

Let $G$ be a finite group, and  $H_1,H_2,\ldots,H_k$ subgroups of $G$.
In general, $G = H_1\cdots H_k$ does not imply $G \equiv H_1\cdots H_k$.
However, when $k=2$, $G=H_1H_2$ implies $G \equiv H_1H_2$.
More precisely, the following lemma can be seen in \cite{Mil13}.

\begin{lemma}\label{lemma:two}
Let $G$ be a finite group, and $H_1$ $H_2$ subgroups of $G$.
If $G = H_1H_2$, then $\mathcal{H}=(H_1,H_2)$ is a uniform group factorization of $G$ with multiplicity $|H_1 \cap H_2|$.
\end{lemma}
\begin{proof}
For any $g\in G$, we write $g=h_1h_2$ for some $h_1\in H_1$ and $h_2\in H_2$.
We then have 
\[ \mult_{\mathcal{H}}^{-1}(g) = \{(h_1y, y^{-1}h_2) \mid y \in H_1\cap H_2\}. \]
This implies the assertion immediately.
\end{proof}

By Lemma \ref{lemma:two}, if $(H_1,H_2)$ is a uniform group factorization of a finite group $G$, and $H_1'$ (resp. $H_2'$) is a maximal subgroup containing $H_1$ (resp. $H_2$), then $(H_1',H_2')$ is also a uniform group factorization.
Thus, we may assume that $H_1$ and $H_2$ are maximal without loss of generality. 

Liebeck, Praeger, and Saxl \cite{LPS90} showed that sporadic simple groups $M_{11}, M_{12}, M_{23}, M_{24}$, $J_2$,  $HS, He, Ru$,  $Suz, Co_1, F_{22}$ have maximal group factorizations of length two (which are in fact uniform group factorizations from Lemma \ref{lemma:two}), and the other sporadic simple groups do not have such factorizations (see also \cite{Giu06}). 
By Lemma \ref{lemma:two}, these sporadic simple groups have a uniform cyclic group factorization with multiplicity greater than $1$. A natural question is whether it is essential that the multiplicity of these cyclic group factorizations be greater than 1. We left as an open problem to find (or prove the inexistence of) a group $G$ such that $G$ has a uniform cyclic group factorization with multiplicity greater than $1$, but does not have one with multiplicity $1$.


\begin{thebibliography}{9999999999}
\bibitem[BSGM02]{BSGM02}
J.~M. Bohli, R.~Steinwandt, M.~I. Gonz{\'a}lez Vasco, and C. Mart{\'i}nez,
{\it Weak keys in MST$_1$},
Des. Codes Cryptogr \textbf{37}(3) (2005), 509--524.

\bibitem[BL21]{BL21}
T. C. Burness, C. H. Li, 
{\it On solvable factors of almost simple groups},
Adv. Math. \textbf{377} (2021), Paper No. 107499, 36 pp.

\bibitem[ATLAS]{ATLAS}
J. H. Conway, R. T. Curtis, S. P. Norton,
R. A. Parker, and R. A. Wilson,
{\it An ATLAS of Finite Groups}, 
Oxford University Press, 1985.

\bibitem[Dix08]{D08} 
J. D. Dixon, 
{\it Generating random elements in finite groups}, 
the electronic journal of combinatorics (2008), R94--R94.

\bibitem[DM96]{DM96} 
J. D. Dixon, B. Mortimer, 
{\it Permutation groups}, 
Graduate Texts in Mathematics, 163. Springer-Verlag, New York, 1996. 
xii+346 pp. 

\bibitem[Gen86a]{Gen86a}
Ts. R. Gentchev, 
{\it Factorizations of the sporadic simple groups},
Arch. Math. (Basel) \textbf{47} (1986), 97--102.

\bibitem[Gen86b]{Gen86b}
Ts. R. Gentchev, 
{\it Factorizations of the groups of Lie type of Lie rank 1 or 2.},
Arch. Math. (Basel) \textbf{47} (1986), 439--499.

\bibitem[Giu06]{Giu06}
M. Giudici, 
{\it Factorisations of sporadic simple groups},
J. Algebra \textbf{304} (2006), 311--323.

\bibitem[GRS03]{GonzalezVasco03}
M.~I. Gonz{\'a}lez Vasco, M.~R{\"o}tteler, and R.~Steinwandt,
{\it On minimal length factorizations of finite groups},
Exp. Math. \textbf{12}(1) (2003), 1--12. 

\bibitem[GS02]{GonzalezVasco02}
M.~I. Gonz{\'a}lez Vasco and R.~Steinwandt,
{\it Obstacles in two public-key cryptosystems based on group factorizations},
Tatra Mt. Math. Publ. \textbf{25} (2002), 23--37.

\bibitem[Gor07]{Gorenstein07}
D.~Gorenstein,
{\it Finite Groups},
Amer. Math. Soc. \textbf{301} (2007).

\bibitem[HLS87]{HLS87}
C. Hering, M. W. Liebeck, and J. Saxl, 
{\it The Factorizations of the Finite Exceptional Groups of Lie type},
J. Algebra \textbf{106}(2) (1987), 517--527.

\bibitem[Hol04]{Holmes04}
P.~E. Holmes,
{\it On minimal factorisations of sporadic groups},
Exp. Math. \textbf{13}(4) (2003), 435--440. 

\bibitem[HL93]{Holt93}
D.~Holt and P.~Rowley,
{\it On products of Sylow subgroups in finite groups},
Arch. Math. (Basel) \textbf{60}(2) (1993), 105--107.

\bibitem[It\^{o}53]{Ito53}
N. It\^{o}, 
{\it On the Factorizations of the Linear Fractional Group ${\rm LF}(2,p^n)$},
Acta Sci. Math. (Szeged) \textbf{15} (1953), 79--84.

\bibitem[LT05]{Lempken05}
W. Lempken and T. van Trung,
{\it On minimal logarithmic signatures of finite groups},
Exp. Math. \textbf{14}(3) (2005), 257--269. 

\bibitem[LTMW09]{Lempken09}
W. Lempken, T. van Trung, S.~S. Magliveras, and W. Wandi,
{\it A public key cryptosystem based on non-abelian finite groups},
J. Cryptology \textbf{22}(1) (2009), 62--74.

\bibitem[LX19]{LX19}
C.H. Li and  B. Xia, 
{\it Factorizations of almost simple groups with a factor having many non-solvable composition factors},
J. Algebra \textbf{528} (2019), 439--473.

\bibitem[LX20]{LX20}
C.H. Li and B. Xia, 
{\it Factorizations of almost simple groups with a solvable factor, and Cayley graphs of solvable groups},
Mem. Amer. Math. Soc. \textbf{279} (2022) no. 432, iv+99 pp.

\bibitem[LPS90]{LPS90}
M.~W. Liebeck, C.~E. Praeger, and J. Saxl, 
\newblock {\it The Maximal Factorizations of the Finite Simple Groups and Their Automorphism Groups}, 
\newblock Mem. Amer. Math. Soc. \textbf{86} (1990) no. 432, iv+151 pp. 

\bibitem[Mag86]{Magliveras86}
S.~S. Magliveras,
{\it A cryptosystem from logarithmic signatures of finite groups},
Proc. of the 29th Midwest Symposium on Circuits and Systems (1986) 972--975.

\bibitem[Mag02]{Magliveras02}
S.~S. Magliveras,
{\it Secret and public-key cryptosystems from group factorizations},
Tatra Mt. Math. Publ. \textbf{25} (2002), 1--12.

\bibitem[MST02]{MagliverasST02}
S.~S. Magliveras, D.~R. Stinson, and Tran van Trung,
{\it New approaches to designing public key cryptosystems using one-way functions and trapdoors in finite groups},
J. Cryptology \textbf{15}(4) (2002), 285--297.

\bibitem[Mil13]{Mil13}
G.~A. Miller, 
{\it The product of two or more groups}, 
Bull. Amer. Math. Soc. \textbf{19} (1913), 303--310.

\bibitem[Rah18]{Rahimipour18}
A.~R. Rahimipour, A.~R. Ashrafi, and A. Gholami,
{\it The existence of minimal logarithmic signatures for some finite simple groups},
Exp. Math. \textbf{27}(2) (2018), 138--146. 

\bibitem[Rah21]{Rah21}
A.~R. Rahimipour, 
{\it Exact Factorizations of Sporadic Simple Groups}, 
Exp. Math. \textbf{30} (2021), 441--446.

\bibitem[Sz\'{e}50]{Sz50}
J. Sz\'{e}p, 
{\it On the structure of groups which can be represented as the product of two subgroups},
Acta Sci. Math. (Szeged)  \textbf{12} (1950), 57--61.

\bibitem[Sz\'{e}51]{Sz51}
J. Sz\'{e}p, 
{\it On factorisable simple groups}, 
Acta Sci. Math. (Szeged)  \textbf{14} (1951), 22.

\bibitem[Zap42]{Zap42}
G. Zappa,
{\it Sulla costruzione dei gruppi prodotto di due dati sottogruppi permutabili tra loro. (Italian)},
Atti Secondo Congresso Un. Mat. Ital., Bologna, 1940, pp. 119–125. Edizioni Cremonese, Rome, 1942.
\end{thebibliography}
\end{document}